\documentclass[12pt]{amsart}
\usepackage[utf8]{inputenc}
\usepackage{amsmath}
\usepackage{amsfonts}
\usepackage{amssymb}
\usepackage{amsthm}
\usepackage{bbm}
\usepackage{hyperref}
\theoremstyle{plain}
\newtheorem{theorem}{Theorem}[section]
\newtheorem*{thma}{Theorem 1.3}
\newtheorem{lemma}[theorem]{Lemma}

\newtheorem{proposition}[theorem]{Proposition}
\theoremstyle{definition}

\theoremstyle{remark}

\newcommand{\Z}{\mathbb{Z}}
\newcommand{\SL}{{\text {\rm SL}}}

\begin{document}

\author{Paul Jenkins}
\address{Department of Mathematics, Brigham Young University, Provo,
UT 84602} \email{jenkins@math.byu.edu}
\author{Kyle Pratt}
\address{Department of Mathematics, Brigham Young University, Provo,
UT 84602} \email{kvpratt@gmail.com}
\title{Interlacing of Zeros of Weakly Holomorphic Modular Forms}

\date \today

\begin{abstract}
We prove that the zeros of a family of extremal modular forms interlace, settling a question of Nozaki. Additionally, we show that the zeros of almost all forms in a basis for the space of weakly holomorphic modular forms of weight $k$ for $\SL_2(\mathbb{Z})$ interlace on most of the lower boundary of the fundamental domain.
\end{abstract}

\maketitle
\section{Introduction and Main Results}
A natural question in studying functions of a complex variable is to determine the location of the zeros of a function; an especially interesting case occurs when the locations of the zeros follow a strong pattern.  Many modular forms have zeros satisfying such properties.  The most well-known such result comes from F. Rankin and Swinnerton-Dyer \cite{RSD70}, who proved that all zeros of the classical Eisenstein series in the standard fundamental domain $\mathcal{F}$ lie on the circular arc $\mathcal{A}= \left\lbrace e^{i \theta} : \frac{\pi}{2} \leq \theta \leq \frac{2\pi}{3} \right\rbrace$ on the lower boundary of $\mathcal{F}$.

When two functions have zeros that lie on the same arc, we say that the zeros of two functions \emph{interlace} if every zero of one function is contained in an open interval whose endpoints are zeros of the other function, and each such interval contains exactly one zero.  Gekeler conjectured that the Eisenstein series $E_k(z)$ satisfy such interlacing properties in \cite{Gek01}, and Nozaki~\cite{Noz08} proved that the zeros of $E_k(z)$ interlace with the zeros of $E_{k+12}(z)$ by improving the bounds used by Rankin and Swinnerton-Dyer.  For the modular function $j_n(z)$ given by the action of the $n$th Hecke operator on $j(z) - 744$, Jermann \cite{Jer12} extended work of Asai, Kaneko, and Ninomiya \cite{AKN97} to prove that the zeros of $j_n(z)$ interlace with the zeros of $j_{n+1}(z)$.  In this paper, we prove interlacing for a family of holomorphic modular forms for $\SL_2(\Z)$, all of whose zeros in $\mathcal{F}$ lie on the arc $\mathcal{A}$.

Denote by $M_k$ the space of holomorphic modular forms of weight $k$, and write $M_k^{!}$ for the larger space of weakly holomorphic modular forms (i.e. poles are allowed at the cusps) of weight $k$.  For such weights $k$, we write $k = 12\ell + k'$, with $k' \in \{0,4,6,8,10,14\}$ and $\ell \in \mathbb{Z}$.
Duke and the first author \cite{DJ08} introduced a canonical basis $\{f_{k, m}(z)\}_{m=-\ell}^\infty$ for $M_k^{!}$ whose elements are defined by
\[f_{k,m}(z) = q^{-m} + O(q^{\ell+1}), \] where, as usual, $q = e^{2\pi i z}$.
They approximated $f_{k,m}(z)$ on the boundary arc $\mathcal{A}$ by a trigonometric function to prove the following theorem locating the zeros of these basis elements.
\begin{theorem}[\cite{DJ08}, Theorem 1] \label{DJTheorem}
If $m \geq |\ell| - \ell$, then all the zeros of $f_{k,m}(z)$ in $\mathcal{F}$ lie on $\mathcal{A}$.
\end{theorem}
\noindent The condition in Theorem \ref{DJTheorem} that $m \geq |\ell| - \ell$ is not sharp, but for all large enough weights $k$,  the zeros of at least one of the $f_{k,m}(z)$ do not all lie on $\mathcal{A}$.  The theorem is generally false when $f_{k, m}$ is a cusp form; one example is the form $f_{132,-9}(z)$.

Getz studied a subset of these basis elements in \cite{Getz04}. We call this family of modular forms ``gap functions'', as they are the holomorphic modular forms with the maximum possible gap in their $q$-expansions.  For any even weight $k \geq 4$, these functions are defined as
\[G_k(z) = f_{k,0}(z) =  1 + O(q^{\ell + 1}). \]
These gap functions have application to the theory of extremal lattices and questions in coding theory; see for example \cite{JR11}.

Our first result proves that the zeros of these gap functions, all of which lie on $\mathcal{A}$, interlace.
\begin{theorem}\label{BigTheorem1}
Let $k$ be an even integer with $k \geq 4$. Then the zeros of $G_k(z)$ interlace on $\mathcal{A}$ with the zeros of $G_{k+12}(z)$ on $\mathcal{A}$.
\end{theorem}
\noindent This theorem settles a question of Nozaki, who suggested that the zeros of these functions ``interlace in some range" (see \cite{Noz08}, Example 1.1).

Additionally, we are able to partially extend our results to the larger class of functions $f_{k,m}(z)$. The result is the following theorem.
\begin{theorem}\label{BigTheorem2}
Let $\epsilon > 0$ and fix $m \geq 0$ (resp. $k \in \Z$). Then the zeros of $f_{k,m}(z)$ interlace with the zeros of $f_{k+12,m}(z)$ (resp. $f_{k,m+1}(z)$) on the arc
\[\mathcal{A}_\epsilon =  \left\lbrace e^{i \theta} : \frac{\pi}{2} < \theta < \frac{2\pi}{3} - \epsilon \right\rbrace \]
for $k$ (resp. $m$) large enough.
\end{theorem}
As the methods used in the proof of Theorem \ref{BigTheorem1} do not entirely apply to the case in which $m$ is nonzero, they do not give interlacing on all of $\mathcal{A}$, although preliminary computations suggest that such interlacing generally holds.  We leave this as an open problem.

\section{Background}\label{Prelims}
We begin by defining some notation.  The standard fundamental domain for $\SL_2(\mathbb{Z})$ is
\[\mathcal{F} = \left\lbrace|z| \geq 1, -\frac{1}{2} \leq \text{Re}(z) \leq 0 \right\rbrace \bigcup \left\lbrace |z| > 1, 0 \leq \text{Re}(z) < \frac{1}{2} \right\rbrace. \]Let $E_0 = 1$, and for even integers $k \geq 4$ let $E_k(z)$ be the usual weight $k$ Eisenstein series
\[E_k(z) = 1 - \frac{2k}{B_k} \sum_{n=1}^{\infty} \sigma_{k-1}(n) q^n \in M_k,  \]
where $B_k$ is the $k$th Bernoulli number and $\sigma_{k-1}(n) = \sum_{d|n} d^{k-1}$. We let \[j(z) = q^{-1} + 744 + 196884q + \cdots\] be the usual modular function in $M_0^!$, and $\Delta(z)$ be the weight 12 cusp form
\begin{equation*}
\Delta(z) = \frac{1}{1728}\left( E_4(z)^3 - E_6(z)^2 \right) = q - 24q^2 + 252q^3 - 1472q^4 + \cdots.
\end{equation*}
With this notation, the basis elements $f_{k, m}(z)$ can be explicitly constructed as \[ f_{k,m}(z) = \Delta^\ell E_{k'} F(j), \]
where $F(j)$ is a polynomial in $j(z)$ of degree $\ell+m$ such that $f_{k, m}(z)$ has the correct Fourier expansion.

The proof of Theorem \ref{DJTheorem} in \cite{DJ08} depends on integrating a generating function for the $f_{k, m}(z)$ to obtain
\[ f_{k,m}(z) = \frac{1}{2\pi i} \oint_C \dfrac{\Delta^\ell(z) E_{k'}(z) E_{14-k'}(\tau)}{\Delta^{1+\ell}(\tau) (j(\tau) - j(z))} r^{-1-m} \, dr, \]
where $r = e^{2\pi i \tau}$ and $C$ is a counterclockwise circle in the $r$-plane centered at 0 with sufficiently small radius. We fix $z = e^{i \theta}$ for some $\theta$ in the interval $I =(\frac{\pi}{2},\frac{2\pi}{3})$ and change variables $r \mapsto \tau$. Then for some $A > 1$ we have
\begin{equation*}
f_{k,m}(z) = \int_{-\frac{1}{2}+iA}^{\frac{1}{2}+iA} \frac{\Delta(z)^\ell}{\Delta(\tau)^{1+\ell}} \frac{E_{k'}(z)E_{14-k'}(\tau)}{j(\tau)-j(z)} e^{-2\pi i m \tau} \, d\tau.
\end{equation*}
We move the contour downward to a height of $A'$. As we do so, each pole in the region
\begin{equation*}
\left\lbrace -\frac{1}{2} < \operatorname{Re}(\tau) < \frac{1}{2}, A' < \operatorname{Im}(\tau) < A \right\rbrace
\end{equation*}
contributes to the value of the integral; these poles occur when $\tau$ is equivalent to $z$ under the action of $\SL_2(\mathbb{Z})$.  If a pole occurs with real part $-\frac{1}{2}$, we modify the contour to include small semicircles in the usual way.

We allow $A'$ to vary depending on $z$. This is to ensure that if $z$ is close to $\frac{\pi}{2}$, then the residue term from $\tau = \frac{z}{z-1}$ does not appear. We choose $A' = .75$ if $\frac{\pi}{2} < \theta \leq 1.9$, picking up residues from $\tau = z$ and $\tau = -\frac{1}{z}$, and we let $A' = .65$ if $\frac{7\pi}{12} \leq \theta < \frac{2\pi}{3}$, picking up an additional residue at $\tau = \frac{z}{z+1}$. The overlap of the two intervals is necessary to obtain interlacing of the zeros.  Applying the Residue Theorem and taking absolute values as in \cite{DJ08}, we obtain
\begin{multline}\label{bigIneq_1.9}
\displaystyle\left|e^{\frac{ik\theta}{2}}e^{-2\pi m\sin{\theta}}f_{k,m}(e^{i\theta}) -2\cos\left(\frac{k\theta}{2} - 2\pi m\cos{\theta}\right)\right| \\ \leq \max_{|x| \leq \frac{1}{2}} e^{-2\pi m (\sin\theta - .75)} \left| \dfrac{\Delta(e^{i \theta})}{\Delta(x +.75i)} \right|^\ell \left| \dfrac{E_{k'}(e^{i \theta}) E_{14-k'}(x+.75i)}{\Delta(x+.75i) (j(x+.75i) - j(e^{i \theta}))} \right|
\end{multline}
when $\frac{\pi}{2} < \theta \leq 1.9$, and
\begin{multline}\label{bigIneq2_7pi12}
\displaystyle\left|e^{\frac{ik\theta}{2}}e^{-2\pi m\sin{\theta}}f_{k,m}(e^{i\theta}) -2\cos\left(\frac{k\theta}{2} - 2\pi m\cos{\theta}\right)
- (-1)^m\frac{e^{-\pi m (2\sin\theta - \tan(\frac{\theta}{2}))}}{(2\cos(\frac{\theta}{2}))^k} \right| \\ \leq \max_{|x| \leq \frac{1}{2}} e^{-2\pi m (\sin\theta - .65)} \left| \dfrac{\Delta(e^{i \theta})}{\Delta(x +.65i)} \right|^\ell \left| \dfrac{E_{k'}(e^{i \theta}) E_{14-k'}(x+.65i)}{\Delta(x+.65i) (j(x+.65i) - j(e^{i \theta}))} \right|
\end{multline}
when $\frac{7\pi}{12} \leq \theta < \frac{2\pi}{3}$.

Note that for any modular form $g$ of weight $k$, the function $e^{\frac{ik\theta}{2}} g(e^{i\theta})$ is real-valued for $\theta \in I$, so the left-hand sides of (\ref{bigIneq_1.9}) and (\ref{bigIneq2_7pi12}) are absolute values of real-valued functions of $\theta$.  Thus, these inequalities give approximations for the modular forms $f_{k,m}(z)$ by the trigonometric function $2\cos\left(\frac{k\theta}{2} - 2\pi m\cos{\theta}\right)$, whose zeros are proved to interlace in the next section.  To prove Theorem \ref{BigTheorem1}, we show that the right-hand sides of equations (\ref{bigIneq_1.9}) and (\ref{bigIneq2_7pi12}) are exponentially decaying functions in the weight $k$, preserving the interlacing for $k$ sufficiently large.  To prove interlacing in the first interval is straightforward. On the other hand, the additional residue term in the second interval shifts the zeros of $G_k(z)$ away from the zeros of the cosine function, necessitating more care.  Computing the interlacing of the zeros of the $G_k(z)$ for all smaller $k$ proves interlacing for all $k \geq 4$.

The proof of Theorem \ref{BigTheorem2} proceeds along similar lines, using the fact that the right-hand sides of (\ref{bigIneq_1.9}) and (\ref{bigIneq2_7pi12}) are also exponentially decaying in $m$.

\section{Interlacing for Cosine Functions}
In this section, we show that the cosine functions obtained in the residue calculation have zeros that interlace.  We define
\begin{equation}
b(\theta) = \frac{k\theta}{2}-2\pi m \cos\theta.
\end{equation}
As we deal with variants of $b(\theta)$ in which $k$ is replaced by $k+12$ or $m$ by $m+1$, we similarly define
\begin{equation}
b_{k+12}(\theta) = \frac{(k+12)\theta}{2}-2\pi m \cos\theta
\end{equation}
and
\begin{equation}
b_{m+1}(\theta) = \frac{k\theta}{2}-2\pi (m+1) \cos\theta.
\end{equation}
When it is clear from context, we write $b_*(\theta)$ to mean either $b_{k+12}(\theta)$ or $b_{m+1}(\theta)$. Note that $\cos(b_*(\theta))$ has one more zero in $I$ than does $\cos(b(\theta))$.

We first prove that the zeros of $\cos(b(\theta))$ and $\cos(b_*(\theta))$ interlace on $I$.

\begin{lemma}\label{CosZeros}
If $m \geq |\ell|-\ell$, then the zeros of $\cos(\frac{k\theta}{2}-2\pi m \cos\theta)$ interlace on $I$ with the zeros of $\cos(\frac{(k+12)\theta}{2}-2\pi m \cos\theta)$ and with the zeros of $\cos(\frac{k\theta}{2}-2\pi (m+1) \cos\theta)$.
\end{lemma}

\begin{proof}
It is clear that to have interlacing the following four conditions are sufficient:
\begin{itemize}
\item\label{first} The first zero in $I$ belongs to $\cos(b_*(\theta))$.
\item\label{second} The last zero in $I$ belongs to $\cos(b_*(\theta))$.
\item\label{fourth} The zeros of $\cos(b_*(\theta))$ and $\cos(b(\theta))$ in $I$ are never equal.
\item\label{third} Between two consecutive zeros of $\cos(b_*(\theta))$ there is exactly one zero of $\cos(b(\theta))$.
\end{itemize}
We prove each of these assertions in turn. It is helpful to work with
\begin{equation*}\label{b_k star}
b_{k+12}(\theta) - 3\pi,
\end{equation*}
because \[\cos( b_{k+12}(\theta) - 3\pi) = -\cos(b_{k+12}(\theta)),\]
so $\cos(b_{k+12}(\theta))$ and $\cos(b_{k+12}(\theta) - 3\pi)$ have the same zeros.  At the endpoints of $I$, we have
\begin{equation}\label{startingValues}
b_{k+12}\left(\frac{\pi}{2}\right) - 3\pi = b_{m+1}\left(\frac{\pi}{2}\right) = b\left(\frac{\pi}{2}\right) = \frac{k\pi}{4}
\end{equation}
and
\begin{equation}\label{endingValues}
b_{k+12}\left(\frac{2\pi}{3}\right) - 3\pi = b_{m+1}\left(\frac{2\pi}{3}\right) = b\left(\frac{2\pi}{3}\right) + \pi.
\end{equation}
By taking derivatives, we see that $b_*(\theta)$ and $b(\theta)$ are monotonically increasing on $I$ for $m \geq |\ell|-\ell$, and that $b'(\theta) < b_*'(\theta)$ for all $\theta \in I$.  Thus, equations (\ref{startingValues}) and (\ref{endingValues}) imply that for all $\theta \in I$, we have
\begin{align}
b(\theta) < b_{k+12}(\theta) - 3\pi < b(\theta) + \pi, \label{BiggerThan1} \\
b(\theta) < b_{m+1}(\theta) < b(\theta) + \pi. \label{BiggerThan2}
\end{align}

Since $b(\frac{\pi}{2})=b_*(\frac{\pi}{2}) = \frac{k\pi}{4}$, the first zeros of $\cos(b(\theta))$ and $\cos(b_*(\theta))$ on $I$ occur when $b(\theta)$ and $b_*(\theta)$ are equal to $\frac{2n+1}{2}\pi$, where $2n+1$ is the first odd integer greater than $\frac{k}{2}$.  Let $\alpha$ be the first zero of $\cos(b_*(\theta))$ on $I$, so that $b_*(\alpha) = \frac{2n+1}{2}\pi$. Then $b(\alpha) < \frac{2n+1}{2}\pi$ by (\ref{BiggerThan1}) and (\ref{BiggerThan2}), so the first zero of $\cos(b_*(\alpha))$ occurs before the first zero of $\cos(b(\alpha))$. This proves the first of the assertions.

The proof of the second assertion follows similarly, though we must make adjustments depending on $k'$.  To see that the zeros cannot be equal, we set $\cos(b(\theta))=\cos(b_*(\theta)) = 0$; equality can only hold for $\theta$ not in $I$.

There can be at most one zero of $\cos(b(\theta))$ between any two consecutive zeros of $\cos(b_*(\theta))$; otherwise,  $b(\theta)$ must increase faster than $b_*(\theta)$ between the zeros, a contradiction. Let $\alpha_1, \alpha_2$ be two consecutive zeros of $b_{m+1}(\theta)$, so that $b_{m+1}(\alpha_1) = b_{m+1}(\alpha_2) - \pi = \frac{n\pi}{2}$ for some integer $n$.  Applying~(\ref{BiggerThan2}) shows that $b(\alpha_1) < \frac{n\pi}{2} < b(\alpha_2)$, so there must exist a point $\beta$ in the interval $(\alpha_1, \alpha_2)$ with $b(\beta) = \frac{n\pi}{2}$ and $\cos(b(\beta)) = 0$.  A similar argument shows that a zero of $\cos(b(\theta))$ appears between every two zeros of $\cos(b_{k+12}(\theta) - 3\pi)$ in $I$.
\end{proof}

In showing interlacing for the $f_{k,m}(z)$, we will need bounds on the distances between the zeros of the approximating cosine functions.  The following proposition gives a preliminary estimate on the distances between zeros.
\begin{proposition}\label{MaxCosZero}
Suppose $m \geq |\ell|-\ell$.  The distance between two consecutive zeros of $\cos(b(\theta))$ is less than or equal to
$\dfrac{2\pi}{k + 2\sqrt{3}m\pi}$.
\end{proposition}
\begin{proof}
This follows from bounding
\begin{equation*}
b{'}(\theta) = \frac{k}{2} + 2\pi m\sin \theta
\end{equation*}
beneath on $I$ by $\dfrac{k}{2} + \pi m \sqrt{3}$.  Note that the condition on $m$ means that this derivative is always positive.  As $\theta$ moves from one zero of $\cos(b(\theta))$ to the next, $b(\theta)$ must increase by $\pi$; the lower bound on the derivative gives an upper bound on the distance between zeros.
\end{proof}
We use Proposition \ref{MaxCosZero} to prove a stronger result.

\begin{lemma}\label{ShortestDist}
Suppose that $m \geq |\ell|-\ell$, and if $b_*(\theta) = b_{k+12}(\theta)$, suppose that $k \geq 0$.  The shortest distance in $I$ between a zero of $\cos(b(\theta))$ and a zero of $\cos(b_*(\theta))$ is between the first zero of $\cos(b_*(\theta))$ and the first zero of $\cos(b(\theta))$, or between the last zero of $\cos(b(\theta))$ and the last zero of $\cos(b_*(\theta))$.
\end{lemma}

The lemma is clearly true when $\cos(b(\theta))$ has only one zero in $I$, and is an immediate consequence of the following proposition.

\begin{proposition}\label{ZerosMoveAround}
Suppose that $m \geq |\ell|-\ell$ and let $\alpha_1, \alpha_2, \alpha_3$ be three consecutive zeros of $\cos(b_{*}(\theta))$ in $I$ and $\beta_1,\beta_2$ be two consecutive zeros of $\cos(b(\theta))$ in $I$ such that $\alpha_1 < \beta_1 < \alpha_2 < \beta_2 < \alpha_3$. If $b_*(\theta) = b_{k+12}(\theta)$, suppose that $k \geq 0$.  Then $\beta_1-\alpha_1 < \beta_2-\alpha_2$, and $\alpha_2-\beta_1 > \alpha_3-\beta_2$.
\end{proposition}

This proposition says that as we examine an increasing sequence of intervals whose endpoints are zeros of $\cos(b_*(\theta))$, the zero of $\cos(b(\theta))$ in each interval moves farther from the left-hand side of the interval and closer to the right-hand side of the interval.

\begin{proof}
We begin in the case where we increase $k$ by 12 by comparing the derivatives $b'(\theta) = \frac{k}{2} + 2\pi m \sin\theta$ and $b_{k+12}'(\theta) = \frac{k}{2} + 6 + 2\pi m \sin\theta = b'(\theta)+6$.  If $m=0$, these derivatives are constant and $b_{k+12}'(\theta_1) > b'(\theta_2)$ for any $\theta_1, \theta_2 \in I$.   Since $b_{k+12}(\theta)$ increases by $\pi$ on $(\alpha_1, \alpha_2)$ and on $(\alpha_2, \alpha_3)$ and $b(\theta)$ increases by $\pi$ on $(\beta_1, \beta_2)$, we conclude that $\alpha_2 - \alpha_1 < \beta_2 - \beta_1$ and $\alpha_3 - \alpha_2 < \beta_2 - \beta_1$, giving the desired inequalities.

If $m > 0$, then $b_{k+12}'(\theta)$ and $b'(\theta)$ are both decreasing functions of $\theta$, as $\sin\theta$ goes from $1$ to $\frac{\sqrt{3}}{2}$ on $I$.  If $b'(\theta) > b_{k+12}'(\theta + \epsilon)$ for some $\epsilon$, then $b_{k+12}'$ has decreased by at least 6 on the interval $(\theta, \theta+\epsilon)$.  Since $b_{k+12}''(\theta) = 2\pi m \cos\theta$ is bounded below by $-m\pi$ on $I$, this can happen only if $\epsilon > \frac{6}{m\pi}$.

Since $\beta_1 < \alpha_2 < \beta_2$, by Proposition \ref{MaxCosZero} it is clear that
\begin{equation*}
\alpha_2 - \beta_1 < \beta_2 - \beta_1 < \dfrac{2\pi}{k + 2\sqrt{3}m\pi} < \dfrac{6}{m\pi}
\end{equation*}
for $k, m$ satisfying $m \geq |\ell|-\ell$.  Thus, we must have $b_{k+12}'(\alpha_2) > b'(\beta_1)$, and since $b'(\theta), b_{k+12}'(\theta)$ are decreasing, it must be true that $b_{k+12}'(\theta_1) > b'(\theta_2)$ for all $\theta_1 \in (\alpha_1, \alpha_2), \theta_2 \in (\beta_1, \beta_2)$.  Since $b_{k+12}(\theta)$ increases by $\pi$ on $(\alpha_1, \alpha_2)$ and $b(\theta)$ increases by $\pi$ on $(\beta_1, \beta_2)$, we conclude that $\alpha_2 - \alpha_1 < \beta_2 - \beta_1$ and $\beta_1 - \alpha_1 < \beta_2 - \alpha_2$.

To prove that $\alpha_2-\beta_1 > \alpha_3-\beta_2$, we show that $\beta_2 - \beta_1 > \alpha_3 - \alpha_2$.  We have \[\alpha_3 - \beta_1 < \alpha_3 - \alpha_1 < \frac{4\pi}{k + 12 + 2\sqrt{3}m\pi}.\]  If $k \geq 0$, this is less than $\frac{6}{m\pi}$, so $b_{k+12}'(\theta_1) > b'(\theta_2)$ for all $\theta_1 \in (\alpha_2, \alpha_3), \theta_2 \in (\beta_1, \beta_2)$.  Noting that $b_{k+12}(\theta)$ and $b_k(\theta)$ increase by $\pi$ on the appropriate intervals as before, we conclude that $\beta_2 - \beta_1 > \alpha_3 - \alpha_2$.

We now handle the case in which we increase $m$ by 1. We begin by noting that
\begin{equation*}
\left|b_{m+1}^{''}(\theta)\right| < (m+1)\pi,
\end{equation*}
and Proposition \ref{MaxCosZero} gives an upper bound on $\alpha_3 - \alpha_1$ of
\begin{equation*}
\alpha_3-\alpha_1 < \dfrac{4\pi}{k + 2\pi \sqrt{3} (m+1)},
\end{equation*}
so the most that $b_{m+1}^{'}(\theta)$ could decrease between $\alpha_1$ and $\alpha_3$ is
\begin{equation*}
\dfrac{-4(m+1){\pi}^2}{k + 2(m+1)\pi \sqrt{3}}.
\end{equation*}
Now $b_{m+1}^{'}(\alpha_1) > b^{'}(\alpha_1) +\pi \sqrt{3}$, and
\begin{equation*}
\pi \sqrt{3} - \dfrac{4(m+1){\pi}^2}{k + 2(m+1)\pi \sqrt{3}} > 0
\end{equation*}
if $m> \frac{-k\sqrt{3}}{2\pi} - 1$, which is always true when $m \geq |\ell|-\ell$.
When this condition on $m$ holds, $b_{m+1}'(\theta_1) > b'(\theta_2)$ for all $\theta_1, \theta_2$ in the interval $(\alpha_1, \alpha_3)$, and both $b_{m+1}(\theta)$ and $b(\theta)$ increase by $\pi$ on the appropriate intervals, so we get the desired inequalities.
\end{proof}

\section{Proof of Theorem \texorpdfstring{\ref{BigTheorem1}}{1.2}}

Recall that Theorem \ref{BigTheorem1} is a statement about $G_k(z) = f_{k,0}(z)$, so we use Lemmas \ref{CosZeros} and \ref{ShortestDist}, when needed, with $m = 0$. We proceed in two cases, depending on the value of $\theta$.

Suppose that $\theta \in (\frac{\pi}{2},1.9]$. From \eqref{bigIneq_1.9} it is clear that $G_k(z)$ may be approximated by $2\cos(\frac{k\theta}{2})$ if we convert the right-hand side into a suitably decreasing function of $k$. Lemma \ref{CosZeros} shows that the zeros of $2\cos(\frac{k\theta}{2})$ interlace, so the zeros of $G_k(z)$ will also interlace if $G_k(z)$ is sufficiently close to $2\cos\left(\frac{k\theta}{2}\right)$.

Consider the quotient of $\Delta$ functions on the right-hand side of \eqref{bigIneq_1.9}. By Proposition 2.2 of \cite{Getz04}, we know that for $\theta \in I$, we have
\[ \left| \Delta(e^{i\theta}) \right| < .00481,\]
and we check computationally that for $|x| \leq \frac{1}{2}$, we have
\[ \left| \Delta(x + .75i) \right| > 0.00721.\]
Together, for $\theta$ and $x$ in the appropriate ranges, these give
\begin{equation}\label{deltaQuot1.9}
\left| \dfrac{\Delta(e^{i \theta})}{\Delta(x +.75i)} \right| < .66713.
\end{equation}
Duke and the first author proved that for $\theta \in I$,
\begin{equation}\label{1.9_est}
\left| e^{\frac{ik \theta}{2}}G_k(e^{i \theta}) - 2\cos\left(\frac{k\theta}{2}\right) \right| < 1.985.
\end{equation}

Taking \eqref{bigIneq_1.9}, \eqref{deltaQuot1.9}, and \eqref{1.9_est} together shows that
\begin{equation}
\left| e^{\frac{ik \theta}{2}}G_k(e^{i \theta}) - 2\cos\left(\frac{k\theta}{2}\right) \right| < 2.97 (.66713)^\ell.
\end{equation}
Using the relation $k = 12\ell + k'$, it is clear that $\ell \geq \frac{k-14}{12}$, so we define $C(k)=2.97 (.66713)^{\frac{k-14}{12}}$. We see that $C(k) \geq 2.97 (.66713)^\ell$.

We now compute how far the zeros of $G_k(e^{i \theta})$ can stray from the zeros of $2\cos\left(\frac{k\theta}{2}\right)$.
Suppose that \[\left| e^{\frac{ik \theta}{2}}G_k(e^{i \theta}) - 2\cos\left(\frac{k\theta}{2}\right) \right| < C\]
for some constant $C<2$, and let $\alpha$ satisfy $2\cos\left(\frac{k\alpha}{2}\right)=0.$ Then a zero of $e^{\frac{ik \theta}{2}}G_k(e^{i \theta})$ appears in the interval $(\alpha - \varepsilon,\alpha + \varepsilon)$, where $\left|2\cos\left(\frac{k(\alpha \pm \varepsilon)}{2}\right)\right| = C.$ To get an upper bound on $\varepsilon$, consider the line intersecting $2\cos\left(\frac{k\theta}{2}\right)$ through the points $(\alpha-\frac{\pi}{k},\pm 2),(\alpha,0),(\alpha+\frac{\pi}{k}, \mp 2).$ The concavity of $2\cos\left(\frac{k\theta}{2}\right)$ near $\alpha$ implies that this line lies between $2\cos\left(\frac{k\theta}{2}\right)$ and the $\theta$-axis. Therefore, if $\beta$ is a point at which the value of the line is $\pm C$, then $|\beta-\alpha| > \varepsilon$. The absolute value of the slope of the line is $\frac{2k}{\pi}$, and it follows that
\[\varepsilon < \dfrac{\pi C}{2k}. \]

Lemma \ref{ShortestDist} gives us a lower bound on the distances between the zeros of $\cos(b(\theta))$ and $\cos(b_*(\theta))$; this lower bound is $\frac{\pi}{k} - \frac{\pi}{k+12}$ when $k' \in \{0,4,8\}$, and $2\left(\frac{\pi}{k} - \frac{\pi}{k+12}\right)$ when $k' \in \{ 6,10,14 \}$.  This can be seen by considering what happens at the endpoints of $I$ for $k' \equiv 0,2$ mod 4.  For instance, if $k \equiv 0 \pmod{4}$, then at $\theta = \frac{\pi}{2}$, we have that $\frac{k\theta}{2} = \frac{k\pi}{4}$ is an integer multiple of $\pi$, and only needs to increase by $\frac{\pi}{2}$ before $\cos(\frac{k\theta}{2})$ has a zero in $I$. When $k \equiv 2 \pmod{4}$, we see that $\frac{k\theta}{2}$ must increase by $\pi$ before a zero occurs.

Replacing $C$ with $C(k)$, we solve the inequality
\begin{equation*}
\dfrac{\pi C(k)}{2k} < \frac{1}{2} \left( \dfrac{\pi}{k}-\dfrac{\pi}{k+12} \right),
\end{equation*}
which is true when $k \geq 118$. This means that when $k \geq 118$, the zeros of $G_k(z), G_{k+12}(z)$ differ from the zeros of $\cos(b(\theta)), \cos(b_*(\theta))$ by an amount which is less than half the minimum distance between zeros of $\cos(b_*(\theta))$ and $\cos(b(\theta))$. The zeros of $G_k(z)$ and $G_{k+12}(z)$ therefore lie in disjoint, interlacing intervals, and must interlace on $(\frac{\pi}{2},1.9]$ for $k \geq 118$.

Now let $\theta \in [\frac{7\pi}{12},\frac{2\pi}{3})$. The method of the previous case must be modified, as we are dealing with a different approximating function for $G_k(z)$, given by
\begin{equation*}
H_k(\theta) = 2\cos\left(\frac{k\theta}{2}\right)+\dfrac{1}{(2\cos(\frac{\theta}{2}))^k}.
\end{equation*}
We require the following lemma.
\begin{lemma}
The zeros of $H_k(\theta)$ interlace with the zeros of $H_{k+12}(\theta)$ on $[\frac{7\pi}{12},\frac{2\pi}{3})$.
\end{lemma}
\begin{proof}
Lemmas 4.1, 4.2, 4.3, 4.4, and 4.5 in \cite{Noz08} prove that the zeros of $H_k(\theta)$ interlace with the zeros of $H_{k+12}(\theta)$; adding $(2\cos(\frac{\theta}{2}))^{-k}$ does not change the order of the zeros. The function $H_k(\theta)$ is closely related to Nozaki's function $F_k(\theta)$, defined as
\[F_k(\theta) = 2\cos\left(\frac{k\theta}{2}\right)+\dfrac{1}{(2\cos(\frac{\theta}{2}))^k} + R_k(\theta).\]
In using Nozaki's lemmas we may simply take $R_k(\theta) = 0$.
\end{proof}

The term $(2\cos(\frac{\theta}{2}))^{-k}$ is monotonically increasing, is very small for smaller $\theta$, and tends rapidly to 1 for $\theta$ close to $\frac{2\pi}{3}$. This residue term shifts the zeros of $G_k(z)$ away from the zeros of $\cos(b(\theta))$, but for large $k$ the effect is negligible unless $\theta$ is very near $\frac{2\pi}{3}$.

As before, we need a lower bound on the distance between the zeros of $H_k(\theta)$ and $H_{k+12}(\theta)$.
An easily adapted lemma from Nozaki (\cite{Noz08}, Lemma 4.1) shows that for a zero $\alpha^*$ of $2\cos(\frac{k\theta}{2})$ and the corresponding zero $\alpha$ of $H_k(\theta)$, we have
\begin{equation}\label{pi3kBound}
|\alpha-\alpha^*|<\frac{\pi}{3k}
\end{equation}
if $\alpha^* \geq \frac{7\pi}{12}$. This fact follows from the observations that $|2\cos(\frac{k}{2}(\alpha^* \pm \frac{\pi}{3k}))|=1$ and $0<(2\cos(\frac{\theta}{2}))^{-k}<1$ for $\theta \in I$. We will use this fact frequently to estimate quantities involving zeros of $H_k(\theta)$.

Let $\alpha$ denote a zero of $H_k(\theta)$ and $\beta$ an adjacent zero of $H_{k+12}(\theta)$. There are two cases to consider: intervals of the type $(\beta, \alpha)$, and intervals of the type $(\alpha, \beta)$. We will obtain lower bounds on the length of intervals of both types.

Consider first the $(\beta, \alpha)$ intervals.  We may view these essentially as intervals defined by zeros of $2\cos(\frac{k\theta}{2})$ and $2\cos(\frac{(k+12)\theta}{2})$, along with some zero shifts due to the presence of the $(2\cos(\frac{\theta}{2}))^{-k}$ term.  By Proposition~\ref{ZerosMoveAround}, the shortest such interval is the first after $\frac{7\pi}{12}$.

We proceed by cases, according to the congruence class of $k \pmod{12}$.
Suppose that $k' = 0$, so that $k = 12\ell$. We solve $\frac{7\pi (k+12)}{24} \leq \frac{\pi(2n+1)}{2}$ to find the smallest $n$ such that $\frac{\pi(2n+1)}{2}>\frac{7\pi}{12}$ is a zero of $2\cos(\frac{(k+12)\theta}{2})$, and similarly find the next zero of $2\cos(\frac{k\theta}{2})$.  If $\ell$ is even, then
the distance between these cosine zeros is
\begin{equation*}
\frac{(7\ell+1)\pi}{k}-\frac{7\pi}{12}=\frac{\pi}{k},
\end{equation*}
while if $\ell$ is odd, then the interval has length
\begin{equation*}
\frac{7\pi}{12}+\frac{2\pi}{k}-\frac{(7(\ell+1)+1)\pi}{k+12}=\frac{2\pi}{k}-\frac{\pi}{k+12}>\frac{\pi}{k}.
\end{equation*}
Therefore, by Proposition \ref{ZerosMoveAround} and \eqref{pi3kBound}, a lower bound on the length of a $(\beta,\alpha)$ interval is given by
\begin{equation*}
\frac{\pi}{k}-\frac{\pi}{3k}-\frac{\pi}{3(k+12)},
\end{equation*}
which is positive for all $k>0$.

This argument works for each value of $k'$. We have the following lower bounds for intervals of $(\beta,\alpha)$ type:
\begin{align}\label{LotsOfLowerBounds}
&\frac{\pi}{k}-\frac{\pi}{3k}-\frac{\pi}{3(k+12)}, \ (k' = 0), \\
&\frac{5\pi}{3k}-\frac{2\pi}{3(k+12)}-\frac{\pi}{3k}-\frac{\pi}{3(k+12)}, \ (k' = 4), \nonumber \\
&\frac{3\pi}{2k}-\frac{\pi}{2(k+12)}-\frac{\pi}{3k}-\frac{\pi}{3(k+12)}, \ (k' = 6), \nonumber \\
&\frac{4\pi}{3k}-\frac{\pi}{3(k+12)}-\frac{\pi}{3k}-\frac{\pi}{3(k+12)}, \ (k' = 8), \nonumber \\
&\frac{7\pi}{6k}-\frac{\pi}{6(k+12)}-\frac{\pi}{3k}-\frac{\pi}{3(k+12)}, \ (k' = 10), \nonumber \\
&\frac{11\pi}{6k}-\frac{5\pi}{6(k+12)}-\frac{\pi}{3k}-\frac{\pi}{3(k+12)}, \ (k' = 14) \nonumber.
\end{align}

The method for handling the $(\beta, \alpha)$ intervals cannot be easily adapted for $(\alpha, \beta)$ intervals, so we use a different approach.  Our general strategy of proof involves the function $H_k(\theta)+H_{k+12}(\theta)$.  If $\beta - \alpha < \frac{\pi}{3k}$, we show that this function is monotonically increasing or decreasing on the interval $(\alpha, \beta)$. We then obtain a lower bound on $|H_{k+12}(\alpha)|+|H_k(\beta)|$, which gives the change in the value of this function over the interval, and use the trivial bound on the derivative of $H_k(\theta)+H_{k+12}(\theta)$ given by
\begin{equation}\label{DerivBound}
\left|\frac{d}{d\theta}(H_k(\theta)+H_{k+12}(\theta))\right| \leq 4k+24
\end{equation}
to find a lower bound for $\beta - \alpha$.  On the other hand, if $\beta - \alpha \geq \frac{\pi}{3k}$, we may simply use $\frac{\pi}{3k}$ as a lower bound.

To see that $H_k(\theta)+H_{k+12}(\theta)$ is monotonic on the interval $(\alpha, \beta)$ when $\beta - \alpha < \frac{\pi}{3k}$, we note that the interval $(\alpha, \beta)$ is contained in the interval $(\alpha^*-\frac{2\pi}{3k},\alpha^*+\frac{2\pi}{3k})$.  On this larger interval, the absolute value of the derivative of $2\cos\left(\frac{k\theta}{2}\right)$ ranges from $\frac{k}{2}$ to $k$.  The absolute value of the derivative of $(2\cos(\theta/2))^{-k}$, on the other hand, is bounded above by its value at an upper bound for the largest possible $\beta$ of $\theta = \frac{2\pi}{3} - \frac{\pi}{k+12} + \frac{\pi}{3(k+12)}$; this value is at most $(.142)k$.  Thus, the derivative of the cosine term dominates in the interval, and $H_k(\theta)$ is monotonic; it follows that $H_k(\theta)+H_{k+12}(\theta)$ is monotonic on $(\alpha, \beta)$.

We now bound $|H_{k+12}(\alpha)|+|H_k(\beta)|$ when $\beta - \alpha < \frac{\pi}{3k}$. There are three cases to consider, based on the value of $k'$ mod 12, since the behavior of $H_k(\theta)$ depends heavily on $k'$. In each case there are two subcases, since for this type of interval, the zeros of $H_k(\theta)$ and $H_{k+12}(\theta)$ both shift to the left or both shift to the right from the zeros of $2\cos(\frac{k\theta}{2})$ and $2\cos(\frac{(k+12)\theta}{2})$.

\subsection*{Case 1: \texorpdfstring{$k' = 0, 6$}{k' = 0, 6}}
Consider first the subcase in which $\cos\left(\frac{k\theta}{2}\right)$ and $\cos\left(\frac{(k+12)\theta}{2}\right)$ are increasing at consecutive zeros, so that the zeros are shifted to the left by adding the extra term. For any $\gamma$ with $\alpha<\gamma<\beta,$ a lower bound on $|H_{k+12}(\alpha)|+|H_k(\beta)|$ is given by $H_k(\gamma)-H_{k+12}(\gamma)>2\cos(\frac{k\gamma}{2})-2\cos(\frac{(k+12)\gamma}{2})$, since $(2\cos(\frac{\theta}{2}))^{-k}-(2\cos(\frac{\theta}{2}))^{-k-12}$ is positive for all $\theta \in I$. By trigonometric identities this is equal to $4\sin(3\gamma)\sin(\frac{(k+6)\gamma}{2})$.

The function $4\sin(3\theta)\sin(\frac{(k+6)\theta}{2})$ has zeros on $I$ at
$\frac{2\pi}{k+6}(\frac{k+6}{3}-n)$, where $n$ is a nonnegative integer. In this case, with the zeros shifted left, $n$ is even. On the other hand, the zeros of $2\cos(\frac{k\theta}{2})$ are at $\theta=\frac{2\pi}{3}-\frac{(2n+1)\pi}{k} \geq \frac{7\pi}{12}$, with
$n \leq \frac{k-12}{24}$. Given this restriction on $n$, it is straightforward to confirm that
\begin{align*}
\frac{2\pi}{k+6}\left(\frac{k+6}{3}-n-1\right)&<\frac{2\pi}{3}-\frac{(2n+1)\pi}{k}-\frac{\pi}{3k}, \\
\frac{2\pi}{k+6}\left(\frac{k+6}{3}-n\right)&>\frac{2\pi}{3}-\frac{(2n+1)\pi}{k+12}+\frac{\pi}{3(k+12)},
\end{align*}
implying that the zeros $\alpha$ and $\beta$ of $H_k(\theta)$ and $H_{k+12}(\theta)$ both lie between the same two zeros of $4\sin(3\theta)\sin(\frac{(k+6)\theta}{2})$.

When the zeros $\alpha,\beta$ of $H_k(\theta)$ and $H_{k+12}(\theta)$ are shifted left, then the zero $\frac{2\pi}{3}-\frac{(2n+1)\pi}{k+12}$ of $\cos\left(\frac{(k+12)\theta}{2}\right)$ is
greater than $\beta$. Because of the parity of $n$, the derivative of $4\sin(3\theta)\sin(\frac{(k+6)\theta}{2})$ is negative at $\frac{2\pi}{3}-\frac{(2n+1)\pi}{k+12}$, implying that the function is positive on all of $(\alpha, \beta)$ and that a lower bound on $4\sin(3\gamma)\sin(\frac{(k+6)\gamma}{2})$ is given by
\begin{equation*}
4\sin\left(3\left( \frac{2\pi}{3}-\frac{(2n+1)\pi}{k+12}\right) \right)\sin\left(\frac{k+6}{2}\left(\frac{2\pi}{3}-\frac{(2n+1)\pi}{k+12} \right) \right),
\end{equation*}
which simplifies to
\begin{equation*}
4\sin\left(\frac{3(2n+1)\pi}{k+12} \right)\sin\left(\frac{(2n+1)(k+6)\pi}{2(k+12)}\right).
\end{equation*}
With the condition $n \leq \frac{k-12}{24}$, we find that
\begin{equation*}
\frac{\pi}{4}<\frac{(2n+1)(k+6)\pi}{2(k+12)}-n\pi<\frac{\pi}{2},
\end{equation*}
and since $n$ is even this implies $\sin(\frac{(2n+1)(k+6)\pi}{2(k+12)})>\frac{\sqrt{2}}{2}$. Therefore, a lower bound on $|H_{k+12}(\alpha)|+|H_k(\beta)|$ when $\alpha, \beta$ are shifted left is given by
\begin{equation*}
\sin\left(\frac{3\pi}{k+12} \right),
\end{equation*}
since $\sin(\frac{3(2n+1)\pi}{k+12})$ is increasing in $n$.

Now suppose that $\alpha,\beta$ are shifted right. Then a lower bound on $|H_{k+12}(\alpha)|+|H_k(\beta)|$ is given by $H_{k+12}(\gamma)-H_k(\gamma)$ for some $\alpha<\gamma<\beta$. Ignoring for the moment the exponential terms, we find a lower bound on $2\cos(\frac{(k+12)\gamma}{2})-2\cos(\frac{k\gamma}{2})$ by arguing as above to replace $\gamma$ by $\frac{2\pi}{3}-\frac{(2n+1)\pi}{k+12}+\frac{\pi}{3(k+12)}$. Inserting this for $\theta$ in the function $-4\sin(3\theta)\sin(\frac{(k+6)\theta}{2})$ and simplifying as before, we have
\begin{equation*}
-4\sin\left(\frac{2(3n+1)\pi}{k+12} \right)\sin\left(\frac{(k+6)(3n+1)\pi}{3(k+12)} \right).
\end{equation*}
Because $n$ is now odd, we find that $-4\sin(\frac{(k+6)(3n+1)\pi}{3(k+12)})>-4\cos(\frac{7\pi}{12})>1$, so this is bounded beneath by $\sin(\frac{2(3n+1)}{k+12})$.

Now we consider the term $(2\cos(\frac{\gamma}{2}))^{-k-12} - (2\cos(\frac{\gamma}{2}))^{-k} = -(2\cos(\frac{\gamma}{2}))^{-k} (1 - (2\cos(\frac{\gamma}{2}))^{-12})$. Since $(2\cos(\frac{\theta}{2}))^{-k}$ is increasing on $I$ and $1-(2\cos(\frac{\theta}{2}))^{-12}$ is decreasing on $I$, an upper bound on the absolute value of $-(2\cos(\frac{\gamma}{2}))^{-k}(1-(2\cos(\frac{\gamma}{2}))^{-12})$ is given by
\begin{equation*}
T_{0,6}(k,n)=\frac{1-\left(2\cos\left(\frac{1}{2}\left(\frac{2\pi}{3}-\frac{(2n+1)\pi}{k} \right)\right)\right)^{-12}}{\left(2\cos\left(\frac{1}{2}\left(\frac{2\pi}{3}-\frac{(2n+1)\pi}{k+12}+\frac{\pi}{3(k+12)}\right)
\right)\right)^{k}}.
\end{equation*}
Since $T_{0,6}(k,n)$ is decreasing in $n$ for fixed $k$, a lower bound on $|H_{k+12}(\alpha)|+|H_k(\beta)|$ is given by
\begin{equation*}
\sin\left(\frac{2(3n+1)}{k+12}\right)-T_{0,6}(k,1).
\end{equation*}
These two lower bounds, along with \eqref{DerivBound}, give a lower bound on $\beta - \alpha$ for $k \equiv 0, 6 \pmod{12}$ of the smaller of $\frac{\pi}{3k}$ and
\begin{equation*}
B_{0,6}(k)=\frac{1}{4k+24}\left(\sin\left(\frac{3\pi}{k+12}\right)-T_{0,6}(k,1) \right).
\end{equation*}
It is not difficult to verify the positivity of $B_{0,6}(k)$. Comparing $B_{0,6}(k)$ with the appropriate quantities from \eqref{LotsOfLowerBounds}, we find that $B_{0,6}(k)$ is indeed a lower bound on the zero distance.

We saw earlier that near zeros of $H_k(\theta)$, we have $|\frac{d}{d\theta}(H_k(\theta))|>\frac{7k}{20}$. If it is true that $\left| e^{ik \theta/2} G_k(e^{i \theta}) - H_k(\theta) \right| < C$, then this gives an upper bound on the distance a zero of $e^{ik \theta/2} G_k(e^{i \theta})$ can travel of $\frac{20C}{7k}.$ Performing calculations as in the case where $\theta \in (\frac{\pi}{2},1.9]$, we may take $C$ to be $2.24(.44)^{(k-6)/12}$. We then solve the inequality
\begin{equation*}
2.24(.44)^{(k-6)/12}\frac{20}{7k}<\frac{1}{2}B_{0,6}(k),
\end{equation*}
which holds for $k \geq 102$.

\subsection*{Case 2: \texorpdfstring{$k' = 4, 10$}{k' = 4, 10}}
This and the following case follow very similarly to the one above. Again we find that $\alpha,\beta$ lie between two zeros of $4\sin(3\theta)\sin(\frac{(k+6)\theta}{2})$, and we handle separately the cases in which $\alpha$ and $\beta$ are shifted left or right. If we define
\begin{equation*}
T_{4,10}(k,n)=\frac{1-\left(2\cos\left(\frac{1}{2}\left(\frac{2\pi}{3}-\frac{(6n+5)\pi}{3k} \right)\right)\right)^{-12}}{\left(2\cos\left(\frac{1}{2}\left(\frac{2\pi}{3}-\frac{(6n+5)\pi}{3(k+12)}+ \frac{\pi}{3(k+12)}\right)\right)\right)^{k}},
\end{equation*}
then a lower bound on zero distance is given by
\begin{equation*}
B_{4,10}(k)=\frac{1}{4k+24}\left(\sin\left(\frac{4\pi}{3(k+12)}\right)-T_{4,10}(k,0) \right).
\end{equation*}
Here $B_{4,10}(k)$ is positive for $k \geq 16$, and comparison with \eqref{LotsOfLowerBounds} shows that $B_{4,10}(k)$ is a lower bound on zero distance over the range of consideration. Similarly as above, we bound from beneath the absolute value of the derivative of $H_k(\theta)$, and find it is bounded by $\frac{7k}{20}$. We find that the inequality
\begin{equation*}
2.24(.44)^{(k-10)/12}\frac{20}{7k}<\frac{1}{2}B_{4,10}(k)
\end{equation*}
holds for $k \geq 128$.

\subsection*{Case 3: \texorpdfstring{$k' = 8, 14$}{k' = 8, 14}}
Proceeding as above, a lower bound on the distance between zeros of $H_k(\theta)$ and $H_{k+12}(\theta)$ is given by
\begin{equation*}
B_{8,14}(k)=\frac{1}{4k+24}\left(\sin\left(\frac{7\pi}{k+12} \right)-T_{8,14}(k,0) \right),
\end{equation*}
where
\begin{equation*}
T_{8,14}(k,n)=\frac{1-\left(2\cos\left(\frac{1}{2}\left(\frac{2\pi}{3}-\frac{(6n+7)\pi}{3k} \right)\right)\right)^{-12}}{\left(2\cos\left(\frac{1}{2}\left(\frac{2\pi}{3}- \frac{(6n+7)\pi}{3(k+12)}+\frac{\pi}{3(k+12)}\right)\right)\right)^{k}}.
\end{equation*}
The inequality here is
\begin{equation*}
2.24(.44)^{(k-14)/12}\frac{20}{7k}<\frac{1}{2}B_{8,14}(k),
\end{equation*}
which holds for $k \geq 98$.

Comparing our results from the two intervals, we find that the zeros of $G_k(z)$ interlace with the zeros of $G_{k+12}(z)$ on the lower boundary of the fundamental domain for $k \geq 128$. We have confirmed computationally that the zeros interlace for $k \leq 140$, and we have an appropriate intersection between our two intervals since Proposition \ref{MaxCosZero} implies that this intersection contains at least two zeros when $k \geq 94$. It follows that the zeros of $G_k(z)$ and $G_{k+12}(z)$ interlace on the lower boundary of the fundamental domain. \hfill $\qed$

\section{Proof of Theorem \texorpdfstring{\ref{BigTheorem2}}{1.3}}
For convenience we restate Theorem \ref{BigTheorem2}.
\begin{thma}
Let $\epsilon > 0,$ and fix $m \geq 0$ (resp. $k \in \Z$). Then the zeros of $f_{k,m}(z)$ interlace with the zeros of $f_{k+12,m}(z)$ (resp. $f_{k,m+1}(z)$) on the arc
\[\mathcal{A}_\epsilon =  \left\lbrace e^{i \theta} : \frac{\pi}{2} < \theta < \frac{2\pi}{3} - \epsilon \right\rbrace \]
for $k$ (resp. $m$) large enough.
\end{thma}

Our proof follows the outlines of the proof above;  the most significant differences involve the lower bounds on the distances between zeros. We take linear approximations to $b(\theta)$ and $b_*(\theta)$ and use those approximations to derive lower bounds on the distance between zeros. We require the hypotheses for Lemma \ref{ShortestDist} to hold; we then need only find such a bound near $\theta=\frac{\pi}{2}$ and $\theta = \frac{2\pi}{3}$.

We first determine the bound for zeros near $\theta = \frac{\pi}{2}$. Taking the first order Taylor series approximation for

\begin{equation*}
b(\theta) = \frac{k\theta}{2} - 2\pi m \cos \theta
\end{equation*}
gives us

\begin{equation*}
L_{k,m}(\theta) = \frac{k\pi}{4} + \frac{k+4m \pi}{2} \left(\theta - \frac{\pi}{2} \right).
\end{equation*}
When we increase $k$ by 12, the linear approximations to $b(\theta)$ and $b_{k+12}(\theta)$ have the same error term.

Write $b(\theta)=L_{k,m}(\theta)-R_{m}(\theta)$. Note that $R_{m}(\theta)$ is increasing and positive on $I$, since $b(\theta)$ is concave down. Let $\alpha_1,\alpha_2$ be the first zeros of $\cos(L_{k+12,m}(\theta)),\cos(L_{k,m}(\theta))$ in $I$, respectively, and let $\beta_1,\beta_2$ be the first zeros on $I$ of $\cos(b_{k+12}(\theta)),\cos(b(\theta))$. We then have, for integers $n_1$ and $n_2$,
\begin{align*}
b_{k+12}(\alpha_1)&=\frac{2n_1+1}{2}\pi-R_m(\alpha_1),\\
b_{k+12}(\beta_1) &= \frac{2n_1+1}{2}\pi,\\
b(\alpha_2) &= \frac{2n_2+1}{2}\pi-R_m(\alpha_2),\\
b(\beta_2) &= \frac{2n_2+1}{2}\pi.
\end{align*}
Now we find the slopes of the lines between $(\alpha_1,b_{k+12}(\alpha_1))$ and $(\beta_1,b_{k+12}(\beta_1))$, $(\alpha_2,b(\alpha_2))$ and $(\beta_2,b(\beta_2))$, and apply the Mean Value Theorem. The slope can be taken to be the value of the derivative at a point in the interval, and by the proof of Proposition \ref{ZerosMoveAround} the derivative of $2\cos(b_*(\theta))$ is greater than the derivative of $2\cos(b(\theta))$ in the appropriate intervals for $k$ large enough. Thus, for large $k$ we have that
\begin{equation*}
\frac{R_m(\alpha_1)}{\beta_1-\alpha_1}>\frac{R_m(\alpha_2)}{\beta_2-\alpha_2},
\end{equation*}
which implies $\beta_2-\alpha_2>\beta_1-\alpha_1$. This in turn implies that $\beta_2-\beta_1 > \alpha_2 - \alpha_1$, so the distance between the zeros of $\cos(L_{k,m}(\theta))$ and $\cos(L_{k+12,m}(\theta))$ is less than the distance between the zeros of $\cos(b(\theta))$ and $\cos(b_{k+12}(\theta))$.  Computing the distance between the zeros of $\cos(L_{k,m}(\theta)),\cos(L_{k+12,m}(\theta))$ we get a lower bound of
\begin{equation*}
\dfrac{\pi}{k + 4m \pi} - \dfrac{\pi}{k+12 + 4m \pi}
\end{equation*}
for the distance between zeros near $\theta = \frac{\pi}{2}$.

The argument for increasing $m$ by 1 is not exactly analogous because the error terms are no longer identical. We use the Taylor series approximation for $b(\theta)$ and use the fact that near $\theta=\frac{\pi}{2}$ we have $b(\theta)$ close to its first-order approximation.

Assume $k>0,$ since the case $k<0$ is similar. Additionally, assume $k \equiv 0 \pmod{4}$; if $k \equiv 2 \pmod{4}$ we find that the lower bound on the zeros is greater than the lower bound when $k \equiv 0 \pmod{4}$. Bounding from beneath the derivative of $b(\theta)$, we see that because $k \equiv 0 \pmod{4}$, the first zero of $2\cos(b(\theta))$ in $I$ is less than $\frac{\pi}{2}+\frac{\pi}{k+2\pi \sqrt{3} m}$. By Taylor's Theorem,
\begin{equation*}
b(\theta) = L_{k,m}(\theta) + \frac{b''(\xi)}{2}\left(\theta-\frac{\pi}{2} \right)
\end{equation*}
for some $\xi \in (\frac{\pi}{2},\frac{\pi}{2}+\frac{\pi}{k+2\pi \sqrt{3} m})$. This gives the inequality
\begin{align*}
L_{k,m}(\theta) &\leq b(\theta) + \pi m \left(\frac{\pi}{k+2\pi \sqrt{3} m} \right)^2 \\
&= b(\theta) + \frac{1}{12\pi m} - \frac{k}{12\sqrt{3}\pi^2 m^2} + \cdots \leq b(\theta) + \frac{1}{m},
\end{align*}
for $m$ large enough with respect to $k$.

Since the first zero of $2\cos(L_{k,m}(\theta))$ in $I$ is at $\frac{\pi(2+k+4\pi m)}{2(k+4\pi m)}$, we see that the first zero of $2\cos(b(\theta))$ in $I$ is less than $\frac{\pi(2+k+4\pi m)}{2(k+4\pi m)}+\frac{2}{m(k+2\pi\sqrt{3}m)}$. Thus, a lower bound on the distance between the first zeros of $2\cos(b(\theta))$ and $2\cos(b_{m+1}(\theta))$ is given by
\begin{align*}
\frac{\pi(2+k+4\pi m)}{2(k+4\pi m)}
-\frac{\pi(2+k+4\pi (m+1))}{2(k+4\pi (m+1))}
-\frac{2}{(m+1)(k+2\pi\sqrt{3}(m+1))},
\end{align*}
which is positive for $m$ large enough with respect to $k$.

We use similar arguments for the lower bound near $\theta = \frac{2\pi}{3}$, and find that the lower bound between zeros, whether we increase $k$ or increase $m$, is given by a decreasing rational function in $k$ and $m$.

Now we pick $\epsilon > 0$, which is fixed for the remainder of the proof. From equations \eqref{bigIneq_1.9} and \eqref{bigIneq2_7pi12} and the proof of Theorem \ref{DJTheorem} we obtain
\[\left|e^{\frac{ik \theta}{2}}e^{-2\pi m\sin{\theta}}f_{k,m}(e^{i\theta}) -2\cos\left(\frac{k\theta}{2} - 2\pi m\cos{\theta}\right) \right|   < 2.97 (.49)^m (.67)^\ell + \frac{e^{-\pi m (2\sin\rho - \tan(\frac{\rho}{2}))}}{(2\cos(\frac{\rho}{2}))^k},\]
where $\rho = \frac{2\pi}{3} - \epsilon$. We do so by comparing the bounds for the two different intervals and choosing our bound to be larger than both of them.  Note that each term of the right side is exponentially decaying in both $m$ and $k$.

Suppose more generally that $|e^{\frac{ik \theta}{2}}e^{-2\pi m\sin{\theta}}f_{k,m}(e^{i\theta}) -2\cos\left(\frac{k\theta}{2} - 2\pi m\cos{\theta}\right) |<D$.  We want to derive an upper bound in terms of $D$ on the distance a zero of $f_{k,m}(e^{i\theta})$ can be from a zero of $2\cos(b(\theta))$. We will do this by bounding from beneath the absolute value of the derivative of $2\cos(b(\theta))$ on intervals around each of its zeros.  If $D$ is small enough, the zeros of $f_{k, m}(e^{i\theta})$ must lie in these intervals, and we may argue as in the proof of Theorem~\ref{BigTheorem1} to obtain a bound involving an exponentially decaying quantity.

Suppose that $2\cos(b(\alpha^*))=0$ for some $\alpha^*$.  We calculate that
\begin{equation*}
\frac{d}{d\theta}\left(2\cos(b(\theta)) \right)=-(k+4\pi m\sin\theta)\sin(b(\theta));
\end{equation*}
using the above bounds, we must bound $|\frac{d}{d\theta}\left(2\cos(b(\theta)) \right)|$ from below on a suitable interval. Trivially bounding the derivative of $2\cos(b(\theta))$, we see that $\frac{\pi}{2(|k|+4\pi m)}$ is an insufficient change in absolute value for $2\cos(b(\theta))$ to reach an extreme. Thus we consider the interval $(\alpha^*- \frac{\pi}{2(|k|+4\pi m)}, \alpha^*+ \frac{\pi}{2(|k|+4\pi m)})$.
Standard formulas give
\begin{equation*}
\sin\left(b\left(\alpha^* \pm \frac{\pi}{2(|k|+4\pi m)}\right)\right)=\sin x\cos y \pm \cos x\sin y,
\end{equation*}
where $x=\frac{k\alpha^*}{2}-2\pi m\cos \alpha^* \cos(\frac{\pi}{2(|k|+4\pi m)})$ and $y=\frac{k\pi}{4(|k|+4\pi m)}+2\pi m\sin \alpha^* \sin(\frac{\pi}{2(|k|+4\pi m)})$. When $k$ or $m$ is large enough, $\cos(\frac{\pi}{2(|k|+4\pi m)})$ is very near 1, so $\cos(x) \approx 0$ and $|\sin(x)| \approx 1$. When $m$ is fixed and $k$ increases, we see that $y$ approaches $\frac{\pi}{4}$, so $|\sin(b(\alpha^* \pm \frac{\pi}{2(|k|+4\pi m)}))|$ is close to $\frac{\sqrt{2}}{2}$. When $k$ is fixed and $m$ increases, $y$ approaches $\frac{\pi \sin \alpha^*}{4}$, and  $|\sin(b(\alpha^* \pm \frac{\pi}{2(k+4\pi m)}))|$ is close to or greater than $\cos(\frac{\pi}{4}) = \frac{\sqrt{2}}{2}$, since $\alpha^*$ may range from $\frac{\pi}{2}$ to $\frac{2\pi}{3}$.
With this bound at the endpoints of our interval, we let $E$ be a positive constant smaller than $\frac{\sqrt{2}}{2}$, such that for all $k,m$ under consideration with $k$ fixed and $m$ increasing (or $m$ fixed and $k$ increasing) and $\theta \in (\alpha^*-\frac{\pi}{2(|k|+4\pi m)},\alpha^*+\frac{\pi}{2(|k|+4\pi m)})$ we have
\begin{equation*}
\left|\frac{d}{d\theta}\left(2\cos(b(\theta)) \right) \right| = |\sin(b(\theta))| |k + 4\pi m \sin\theta| > E\left(k+2\sqrt{3}\pi m\right).
\end{equation*}

Letting $M(k,m)$ be the minimum of the four lower bounds on the distance between zeros we calculated above and replacing $D$ with our exponential quantities from before, we find that the zeros interlace when
\begin{equation*}
\frac{1}{E(k+2\sqrt{3}\pi m)} \left( 2.97 (.49)^m (.67)^\ell + \frac{e^{-\pi m (2\sin\rho - \tan(\frac{\rho}{2}))}}{(2\cos(\frac{\rho}{2}))^k}  \right) < \frac{1}{2} M(k,m).
\end{equation*}
Because the left-hand side has exponential decay in $k$ and $m$ while the right-hand side is a rational function of $k$ and $m$, we see that the inequality holds for all but finitely many $k$ and $m$, so the zeros of $f_{k,m}(z)$ interlace with the zeros of $f_{k+12,m}(z)$ or $f_{k,m+1}(z)$ on $(\frac{\pi}{2},\frac{2\pi}{3} - \epsilon)$ for $k,m$ sufficiently large. \hfill $\qed$

\bibliographystyle{amsplain}

\end{document}